\documentclass[a4paper,11pt]{scrartcl}
\addtokomafont{disposition}{\rmfamily}

\usepackage[left=2.5cm,right=2.5cm,top=2.5cm,bottom=2.5cm,includefoot]{geometry}

\usepackage{amsthm}
\usepackage{amsmath}
\usepackage{amssymb}
\usepackage{mathrsfs} 
\usepackage{graphicx}
\usepackage[%
  natbibapa
]{apacite}

\usepackage[dvipsnames]{xcolor}
\usepackage{stackengine}
\usepackage{placeins}
\usepackage{arydshln}
\usepackage[colorlinks=true,linkcolor=blue, citecolor=blue, urlcolor=blue, linktocpage=true]{hyperref} 
\usepackage{bbm}
\usepackage{textcomp}
\usepackage{enumitem}
\usepackage{mathtools}
\usepackage{tabularx}
\usepackage{booktabs}
\usepackage{rotating}

\usepackage{color}
\usepackage{framed}
\usepackage{comment}
\definecolor{shadecolor}{gray}{0.9}

\usepackage{dsfont} 
\usepackage{caption} 
\usepackage{subcaption} 
\usepackage{graphicx}
\usepackage{pdfpages}
\usepackage{url}
\usepackage[english]{babel}
\usepackage{setspace}
\usepackage{dashrule}
\usepackage{tikz}
\usetikzlibrary{automata, positioning, arrows}
\tikzset{
        ->,  
        node distance=5.5cm, 
        every state/.style={thick, fill=gray!10}, 
        initial text=$ $, 
        }

\specialcomment{extra}{\begin{shaded}}{\end{shaded}}

\theoremstyle{plain}  
\newtheorem{theorem}{Theorem}[section] 
\newtheorem{lemma}[theorem]{Lemma}

\theoremstyle{definition}

\newtheorem{example}[theorem]{Example}

\newtheoremstyle{assumption}
{3pt}
{3pt}
{}
{}
{\bf}
{.}
{.5em}
{\thmname{#1} (\thmnote{#3}\thmnumber{#2})}

\theoremstyle{assumption}

\theoremstyle{remark} 
\newtheorem{remark}[theorem]{Remark}

\newcommand{\los}{\mathrm{L}}%
\newcommand{\blos}{\overline{\mathrm{L}}}%

\newcommand{\iden}{\mathrm{V}}%
\renewcommand\theta{\vartheta}

\newcommand{\N}{\ensuremath{\mathbb{N}}}%
\newcommand{\E}{\ensuremath{\mathbb{E}}}%

\newcommand{\Pb}{\ensuremath{\mathbb{P}}}%
\newcommand{\rP}{\ensuremath{\mathrm{P}}}%
\newcommand{\rp}{\ensuremath{\mathrm{p}}}%
\newcommand{\rQ}{\ensuremath{\mathrm{Q}}}%
\newcommand{\rS}{\ensuremath{\mathrm{S}}}%

\newcommand{\R}{\ensuremath{\mathbb{R}}}%
\newcommand{\Q}{\ensuremath{\mathbb{Q}}}%
\newcommand{\Aa}{\ensuremath{\mathcal{A}}}%
\newcommand{\B}{\ensuremath{\mathcal{B}}}%
\newcommand{\F}{\ensuremath{\mathcal{F}}}%
\newcommand{\cP}{\ensuremath{\mathcal{P}}}%
\newcommand{\cQ}{\ensuremath{\mathcal{Q}}}%

\newcommand{\dd}{\ensuremath{{\mathrm d}}}%

\newcommand{\ent}{\mathrm{ent}}

\newcommand{\ES}{\mathrm{ES}}
\newcommand{\VaR}{\mathrm{VaR}}

\newcommand{\ind}{\mathds{1}}
\DeclareMathOperator{\CRPS}{CRPS}
\DeclareMathOperator{\LOGS}{LOGS}

\renewcommand\theta{\vartheta}

\newcommand{\odom}{O}
\newcommand{\sodom}{\mathcal O}

\newcommand{\adom}{A}

\renewcommand{\bf}{\normalfont \bfseries}

\def\be{\begin{equation} \label}
\def\ee{\end{equation}}

\usepackage[colorinlistoftodos,textsize=tiny]{todonotes}
\newcommand{\Comments}{1}
\newcommand{\mynote}[2]{\ifnum\Comments=1\textcolor{#1}{#2}\fi}
\newcommand{\mytodo}[2]{\ifnum\Comments=1%
  \todo[linecolor=#1!80!black,backgroundcolor=#1,bordercolor=#1!80!black]{#2}\fi}

\ifnum\Comments=1               
\fi

\ifnum\Comments=1               
\fi

\begin{document}
%
%
%
\title{Measurability of functionals and of ideal point forecasts}
\author{Tobias Fissler\thanks{Vienna University of Economics and Business (WU), Department of Finance, Accounting and Statistics, Welthandelsplatz 1, 1020 Vienna, Austria, 
		e-mail: \href{mailto:tobias.fissler@wu.ac.at}{tobias.fissler@wu.ac.at}
		} 
		\and 
		Hajo Holzmann\thanks{Philipps-Universit\"at Marburg,
		Fachbereich Mathematik und Informatik, 
		Hans-Meerwein-Stra\ss e, 35043 Marburg, Germany,
		e-mail: \href{holzmann@mathematik.uni-marburg.de}{holzmann@mathematik.uni-marburg.de} }
		}

\maketitle

\begin{abstract}
\noindent
\textbf{Abstract.}
The ideal probabilistic forecast for a random variable $Y$ based on an information set $\F$ is the conditional distribution of $Y$ given $\F$. In the context of point forecasts aiming to specify a functional $T$ such as the mean, a quantile or a risk measure, the ideal point forecast is the respective functional applied to the conditional distribution. 
This paper provides a theoretical justification why this ideal forecast is actually a forecast, that is, an $\F$-measurable random variable. To that end, the appropriate notion of measurability of $T$ is clarified and this measurability is established for a large class of practically relevant functionals, including elicitable ones. 
More generally, the measurability of $T$ implies the measurability of any point forecast which arises by applying $T$ to a probabilistic forecast.
Similar measurability results are established for proper scoring rules, the main tool to evaluate the predictive accuracy of probabilistic forecasts.
\end{abstract}



%
%
%

\section{Introduction}

A typical aspect of decision making, be it in business, politics, or private life, is that decisions should account for unknown or future events.
Hence, these decisions commonly base on predictions or forecasts for these events.
Fortunately, one regularly has at least partial information about the events of interest, such as  regressors or feature variables in a cross-sectional setting or past observations of the event in a time series framework, which help to improve the predictions.
Expressed in terms of mathematical statistics, 
 the event of interest is a random element $Y$ on a probability space which often attains values in the real numbers, and the given information is modelled as a sub-$\sigma$-algebra $\F$ of the $\sigma$-algebra  of events. 

It has been argued in the literature that probabilistic forecasts, taking the form of probability measures, distributions, or densities, are most informative and should be preferred, since they directly express the uncertainty about the future event \citep{Dawid1984, GneitingETAL2007}. The \emph{ideal} or truthful probabilistic forecast of $Y$ based on $\F$ is the conditional distribution of $Y$ given $\F$, denoted by $\Pb_{Y|\F}$. 
``However, practical situations may require single-valued point forecasts, for reasons of decision making, reporting requirements, or communications, among others'' \citep{GneitingKatzfuss2014}. 
In such situations, one summarises the uncertainty in a typically real-valued functional $T$ of the distribution of $Y$ such as the mean, a quantile, an expectile, or a law-determined risk measure such as Expected Shortfall.
Then, the ideal or truthful point forecast for $Y$ based on $\F$ is the functional $T$ applied to the conditional distribution of $Y$ given $\F$, that is, $T(Y|\F) := T(\Pb_{Y|\F})$ \citep{NoldeZiegel2017}.
This concatenation also turns out to be the main idea in building dynamic or conditional risk measures \citep{Weber2006}.
 
Ideal forecasts based on $\F$ need to be $\F$-measurable random elements, or random variables, for the case of point forecasts. 
On the one hand, this ensures that they actually only exploit the information in $\F$ and no additional sources. From a more technical perspective, on the other hand, measurability ensures that one can assign probabilities to events induced by forecasts and, more generally, that forecasts are amenable to statistical analysis. 

This paper studies when an ideal point forecast based on $\F$, or more generally, a forecast which arises by applying the functional to a probabilistic forecast based on $\F$, is indeed $\F$-measurable. To this end, the appropriate notion of measurability of the target functional $T$, viewed as a map from a class of probability measures to the reals, is clarified. Moreover, we establish corresponding measurability results for many practically relevant examples, including the class of elicitable ones, which can be written as the minimiser of an expected loss function (Theorem \ref{thm:main}).
Our results provide a formal justification for the $\F$-measurability of ideal $\F$-based point forecasts which seems to have been taken for granted in large parts of the forecast evaluation literate, sometimes tacitly, sometimes with somewhat incomplete arguments, or which has been established under restrictive assumptions \citep{Weber2006, HolzmannEulert2014, NoldeZiegel2017, Pohle2020, FisslerHoga2021, GneitingResin2021, FFHR2021, HogaDimitriadis2021}.
From a regression or machine learning perspective where $Y$ is observed together with explanatory variables $X$, our results justify the measurability of the oracle regression function $x\mapsto T(Y|X=x)$ beyond the classical situation of mean regression to quantile regression \citep{Koenker2005}, in expectile regression \citep{NeweyPowell1987} or in Expected Shortfall regression \citep{DimiBayer2019}.
In the field of sensitivity analysis, measuring the information content with score improvements also calls for the measurability results established in our paper \citep{BorgonovoETAL2021, FisslerPesenti2022}.

In Section \ref{sec:constscore}, we provide the formal definitions and concepts of the article. The appropriate concept of measurability of functionals is discussed in Section \ref{sec:mesurability of stat func} which leads to the $\F$-measurability of the ideal $\F$-based forecast. 
We show measurability of important functionals such as moments or weighted quantiles and discuss the connection to continuity and robustness results discussed in statistics and finance. Measurability of elicitable functionals is discussed and established in Section \ref{sec:elicitable}.
Finally, we provide sufficient conditions for the measurability of scoring rules in Section \ref{sec:scoring rules}. They are used to measure the predictive accuracy of probabilistic forecasts, and  
map a pair of a predictive distribution $\rP$ and an observation $y$ to the real number $\rS(\rP,y)$. Hence, technically speaking, for fixed $y$ they constitute a real-valued statistical functional, which aligns them with the main topic of this paper. 
Regarding the application, the measurability of scoring rules formally justifies common statistical practices such as computing Diebold--Mariano tests which test predictive dominance \citep{DieboldMariano1995}.

\section{Statistical functionals and forecasting}
\label{sec:constscore}

Real-valued statistical functionals map probability measures on the  Borel-$\sigma$-algebra of events of the observation domain to real numbers. More formally, the \emph{observation domain} $\odom \subseteq \R^d$ is a  Borel-measurable subset of a Euclidean space $\R^d$, where we will mainly consider $d=1$,  with \emph{Borel-$\sigma$-algebra} $\sodom$ of $\odom$. A statistical functional on some family $\cP$ of (Borel-)probability measures on $(\odom, \sodom)$ is a map $T: \cP \to \R$ from $\cP$ to the real numbers.

A prediction space setting \citep{GneitingRanjan2013} consists of a probability space $(\Omega, \Aa, \Pb)$ together with a random variable $Y : \Omega \to \odom$, the response variable modelling the quantity of interest, and a sub-$\sigma$-algebra $\F$ of $\Aa$, the information set on which the forecast is based. $\F$ can be generated, for example, by observable explanatory variables, often called, regressors or features, which can also contain past observations of $Y$ in a time series setting.
Then, the conditional distribution $\Pb_{Y | \F}$ of $Y$ given $\F$ is the \emph{ideal probabilistic} forecast. If a real-valued quantity shall be issued as a forecast, it is essential to specify which statistical functional $T$ of the conditional distribution $\Pb_{Y | \F}$ is the target.   

A general forecast for $T$ of $\Pb_{Y | \F}$ is an $\F$-measurable random variable, and the ideal forecast is $T(Y|\F) : = T\big(\Pb_{Y | \F} \big)$. For $T(Y|\F)$ to be defined, we first require that conditional distributions in $\Pb_{Y | \F}$ are contained in $\cP$, the domain of $T$. Second, we require the $\F$-measurability of the forecast $T(Y|\F)$.
Our aim is to provide conditions on $T$ which guarantee that, when well-defined, $T(Y|\F)$ is automatically a forecast, that is, $\F$-measurable, and to show that these conditions are satisfied by essentially all statistical functionals which arise in applications. 

To proceed, let us recall the formal definition of the conditional distribution $\Pb_{Y | \F}$. A \emph{Markov kernel} from $(\Omega, \F)$ to $(\odom, \sodom)$ is a map $\kappa: \Omega \times \sodom \to [0,1]$ such that 
\begin{enumerate}
	\item
	for each $\omega \in \Omega$, the map $B \mapsto \kappa(\omega, B)$, $B \in \sodom$, is a probability measure on $(\odom, \sodom)$,
	\item
	for each $B \in \sodom$, the map $\omega \mapsto \kappa(\omega, B)$, $\omega \in \Omega$, is $\F - \B[0,1]$-measurable. 
\end{enumerate}

A \emph{regular version} of the conditional probability of $Y$ given $\F$, denoted by $\Pb_{Y | \F}$, is a Markov kernel from $(\Omega, \F)$ to $(\odom, \sodom)$ such that for each $F \in \F$ and $B \in \sodom$, 
\[ \Pb\big(F \cap \{ Y \in B\} \big) = \int_F \Pb_{Y | \F}(\omega; B)\,  \dd \Pb(\omega).\]
Recall that $\Pb_{Y | \F}$ is unique only up to almost sure equality.

When forecasting the functional $T$ for $Y$ based on $\F$, we aim at the map
\begin{equation}\label{eq:functionalcond}
	\omega \mapsto T\big(\Pb_{Y | \F}(\omega;\cdot) \big) = : T(Y|\F)(\omega), \qquad \omega \in \Omega,
\end{equation} 
and we shall assume that the version $\Pb_{Y | \F}$ can be and is chosen such that $\Pb_{Y | \F}(\omega;\cdot) \in \cP$ for all $\omega \in \Omega$. If $\kappa$ is a probabilistic forecast based on $\F$, that is, a Markov kernel from $(\Omega, \F)$ to $(\odom, \sodom)$ for which $\kappa(\omega, \cdot) \in \cP$ for $\omega \in \Omega$, we may form the resulting point forecast by
$$\omega \mapsto T\big(\kappa(\omega;\cdot) \big), \qquad \omega \in \Omega.$$
To show that $T(Y|\F)$ or, more generally, that $T\big(\kappa(\cdot;\cdot) \big)$ is a forecast, that is  $\F$-measurable, it suffices to find a $\sigma$-algebra $\Aa(\cP)$ such that (a) $\Pb_{Y | \F}$  or $\kappa$, considered as a map from $\Omega$ to $\cP$, is $\F - \Aa(\cP)$-measurable, and (b) to show that $T : \cP \to \R$ is $\Aa(\cP) - \B$-measurable, where $\B$ is the Borel-$\sigma$-algebra of $\R$.

\section{Measurability of functionals}
\label{sec:mesurability of stat func}

We shall choose $\Aa(\cP)$ as the projection $\sigma$-algebra. More formally, let 
$\cQ$ be the family of all probability measures $\rP$ on $(\odom, \sodom)$, and 
for each $B \in \sodom$, consider the evaluation map 
\[ \pi_B(\rP) = \rP(B), \qquad \rP \in \cQ.\]
Then, on $\cQ$ we can consider the smallest $\sigma$-algebra $\Aa(\cQ)$ which makes all evaluation maps $\Aa(\cQ) - \B[0,1]$-measurable. For a subset $\cP \subseteq \cQ$ we denote by $\Aa(\cP) = \{ \cP \cap A \mid A \in \Aa(\cQ)\}$ the trace $\sigma$-algebra.
Note that $\Aa(\cP)$ is also the $\sigma$-algebra generated by the restrictions of the evaluation maps $\pi_B$ to $\cP$, which follows from the simple fact that for any $C\in \B[0,1]$ and any $B\in \sodom$ we have $\pi_{B|\cP}^{-1}(C)  = \cP \cap \pi_B^{-1}(C)$.
Moreover, $\Aa(\cP)$ can be regarded as the Borel-$\sigma$-algebra induced by the topology of setwise convergence.
%
\begin{lemma}
\label{lemma:Markov kernels}
	Let $\kappa: \Omega \times \sodom \to [0,1]$ be a map such that for each $\omega\in \Omega$, $\kappa(\omega, \cdot) \in \cP$. Then $\kappa$ is a Markov kernel from $(\Omega, \F)$ to $(\odom, \sodom)$ if and only if the map $\omega \to \kappa(\omega, \cdot)$ from $\Omega$ to $\cP$ is $\F - \Aa(\cP)$-measurable. 
\end{lemma} 	
\begin{proof}
	$\kappa$ is a Markov kernel if and only if for each $B \in \sodom$, the map $\omega \to \kappa(\omega, B) = \pi_B(\kappa(\omega, \cdot))$ is $\F - \B[0,1]$-measurable, which is the case if and only if $\omega \to \kappa(\omega, \cdot)$ is $\F - \Aa(\cP)$-measurable. 
\end{proof}

Concerning the measurability of the forecasts $T(Y|\F)$ of functionals $T$, we observe the following. 
\begin{lemma}
\label{lemma:ideal forecast}
	If the functional $T: \cP \to \R$ is $\Aa(\cP) - \B$-measurable, then $T(Y|\F)$ defined in \eqref{eq:functionalcond} is $\F - \B$-measurable. 
\end{lemma}
\begin{proof}
This is clear since by Lemma \ref{lemma:Markov kernels}, the conditional distribution $\Pb_{Y | \F}$ is $\F - \Aa(\cP)$-measurable as a map from $\Omega$ to $\cP$, and $T$ is $\Aa(\cP) - \B$-measurable by assumption. 
\end{proof}
Similarly to Lemma \ref{lemma:ideal forecast}, the general forecast $\omega \mapsto T\big(\kappa(\omega;\cdot) \big)$, $\omega \in \Omega$, where $\kappa$ is a Markov kernel from $(\Omega, \F)$ to $(\odom, \sodom)$, is $\F - \B$-measurable.
Thus, to clarify the measurability of $T\big(\kappa(\cdot;\cdot) \big)$ and of $T(Y|\F)$ in \eqref{eq:functionalcond} it suffices to investigate $T$. 
Before we turn to a general result for elicitable functionals, let us investigate several important examples. 
In the following we identify any probability distribution $\rP\in\cQ$ with its distribution function $F = F_{\rP}$ defined by $F(x) = \rP(\odom \cap (-\infty, x])$, $x \in \odom$.

\begin{example}[Moments]\label{ex:moments}
Let $h: \odom \to \R$ be a measurably function. 
Let $\cP_h$ be a family of probability measures on $(\odom, \sodom)$ satisfying $\int_\odom \, |h(y)|\, \dd \rP(y) < \infty$, and consider the mean functional
\[ T_h(\rP) = \int_\odom \, h(y)\, \dd \rP(y), \qquad \rP\in\cP_h.\]
Then $T_h$ is $\Aa(\cP_h)-\B$-measurable.

Indeed, if $h(y) = \ind_B(y)$ for some $B \in \sodom$, then $T_h = \pi_B$ and the claim follows from the measurability of the evaluation map.
If $h$ is a simple function, the resulting functional will be a finite linear combination of evaluation maps, and hence also measurable. For $h\geq 0$, there is a sequence of non-negative simple functions $(h_n)$ with $h_n \uparrow h$ pointwise, and then $T_{h_n}(\rP) \uparrow  T_{h}(\rP)$ for every $\rP$. Hence, as a pointwise limit of the measurable functions $T_{h_n}$, $T_h$ is also measurable. For a general $h$, the measurability follows by the decomposition $T_h = T_{h^+} - T_{h^-}$. 
\end{example}

Example \ref{ex:moments} directly yields the measurability of Borel-measurable functions of multiple moments. This yields, for example, the measurability of the variance, the skewness, the kurtosis, or of the Sharpe ratio.

\begin{example}[Quantiles]\label{ex:quantiles}
	Let $\odom = \R$, 
	fix $\alpha \in [0,1]$ and let 
	\begin{align*}
	T(F) = q_\alpha^-(F)
	& = \inf\{ x \in \R \mid F(x) \geq \alpha\} \in \bar \R = \R\cup\{-\infty, \infty\}, \qquad F \in \cP,
	\end{align*}
	be the lower $\alpha$-quantile of $F$, which is its essential supremum if $\alpha=1$.
	Recall that $ \B(\bar \R) = \{A\cup E\mid A\in\B(\R), \ E\subseteq \{-\infty, \infty\}\}$. 
	
	To show that $T$ is $\Aa(\cP) - \B(\bar \R)$-measurable,
	consider the family of evaluation maps $\pi_{(-\infty,x]}(F) = F(x)$, for $F \in \cP$, $x \in \R$. Each $\pi_{(-\infty,x]}$ is $\Aa(\cP) - \B([0,1])$-measurable. 

	Now, given $p \in \R$, the lower $\alpha$-quantile $q_\alpha^-(F)$ of $F$ is strictly larger than $p$, $q_\alpha^-(F) > p$, if and only if $\alpha > F(p)$. Therefore 
	\begin{align*}
		\{F \in \cP \mid  q_\alpha^-(F) > p\} & = \{F \in \cP  \mid  \pi_{(-\infty,p]}(F)  < \alpha\},
	\end{align*}
	showing measurability of the lower $\alpha$-quantile, where we exploit the fact that it suffices to show measurability on a generator of $\B(\bar \R)$.
	
	Similar considerations yield the $\Aa(\cP) - \B(\bar \R)$-measurability of the upper $\alpha$-quantile, $\alpha\in[0,1]$,
	\begin{align*}
	T(F) = q_\alpha^+(F) 
	& = \sup\{ x \in \R \mid \lim_{t\uparrow x} F(t) \leq \alpha\} \in \bar \R, \qquad F \in \cP,
	\end{align*}
	where $q_0^+(F)$ corresponds to the essential infimum of $F$.
	For $\alpha\in(0,1)$, we can also consider the $\alpha$-quantile, $q_\alpha$, as an interval-valued functional $q_\alpha(F) = [q_\alpha^-(F), q_\alpha^+(F)]\subset \R$. Then, the previously established measurability results yield that $q_\alpha$ is an \emph{Effros measurable} closed-valued multifunction in the sense that for all open sets $G\subseteq\R$, $\{F\in\cP\mid q_\alpha(F)\cap G \neq \emptyset\}\in \Aa(\cP)$, see Definition 1.3.1 in \cite{Molchanov2017}. Therefore, the ideal forecast $q_\alpha(Y|\F)$ is a random closed set in the sense of Definition 1.1.1 in \cite{Molchanov2017}.
	This complements the measurability argument of the quantile provided recently in \cite{CastroETAL2021}.
\end{example}
In the risk management literature, the lower $\alpha$-quantile is known as Value at Risk at level $\alpha$, $\VaR_\alpha$.
The following example is concerned with another important quantitative risk measure, the Expected Shortfall.

\begin{example}[Weighted averages of quantiles]\label{ex:average quantiles}
Let $w\colon[0,1]\to[0,\infty)$ be a measurable weight function such that the left-sided and the right-sided limits exists for all $x\in[0,1]$. 
Consider the functional 
\begin{equation*}
T_w(F) = \int_0^1 q_\gamma^-(F)w(\gamma)\, \dd \gamma \in \bar \R, \qquad F\in \cP_w,
\end{equation*}
where $\cP_w$ is the domain such that $T_w$ is well-defined. The measurability of $q_\gamma^-$ from Example \ref{ex:quantiles} and a Riemann approximation argument yield that $T_w$ is $\Aa(\cP_w) - \B(\bar \R)$-measurable.

This directly yields the measurability of the lower and upper Expected Shortfall (by choosing $w = \ind_{[0,\alpha]}/\alpha$ or $w = \ind_{[\alpha,1]}/(1-\alpha)$) and of the Range Value at Risk \citep{ContDeguestETAL2010, FisslerZiegel2021} (by choosing $w= \ind_{[\alpha,\beta]}/(\beta - \alpha)$ for $0<\alpha<\beta<1$).
\end{example}
To treat the next Example \ref{ex:CoVaR}, we first need to establish the following technical lemma.
\begin{lemma}\label{lem:technical}
Let $T\colon \cP\to\R$ be an $\Aa(\cP) - \B(\R)$-measurable functional. Then the composed evaluation maps $\cP\to[0,1]$,
\begin{equation}\label{eq:composed}
F\mapsto \pi_{(T(F),\infty)}(F), 
\qquad F\mapsto \pi_{(-\infty,T(F))}(F), 
\qquad F\mapsto \pi_{\{T(F)\}}(F)
\end{equation}
are $\Aa(\cP) - \B([0,1])$-measurable.
\end{lemma}

\begin{proof}
For the first map in \eqref{eq:composed}, let $p\in[0,1]$ and consider
\begin{align*}
\{F\in \cP\mid \pi_{(T(F),\infty)}(F) \le p\}
&= \{F\in \cP\mid F(T(F))\ge 1- p\} 
=  \{F\in \cP\mid T(F)\ge q_{1- p}^-(F)\} \\
&=  \{F\in \cP\mid (T -q_{1- p}^-)(F) \ge 0\}
\in \Aa(\cP).
\end{align*}
The last assertion comes from the fact that $q_{1-p}^-$ is $\Aa(\cP) - \B(\R)$-measurable and that differences of measurable functionals are also measurable.\\
For the second map in \eqref{eq:composed}, the arguments are similar. Using the shorthand $F(x-) := \lim_{t\uparrow x}F(t)$ for the left-sided limit, we obtain for $p\in[0,1]$
\begin{align*}
\{F\in \cP\mid \pi_{(-\infty,T(F))}(F) \le p\}
&= \{F\in \cP\mid F(T(F)-)\le p\} 
= \{F\in \cP\mid T(F)\le q_p^+(F) \} \\
&=  \{F\in \cP\mid (T - q_p^+)(F)\le 0 \} \in \Aa(\cP).
\end{align*}
Finally, note that $\pi_{\{T(F)\}}(F) = 1-\pi_{(T(F),\infty)}(F) - \pi_{(-\infty,T(F))}(F)$.
\end{proof}

\begin{example}\label{ex:CoVaR}
Conditional Value at Risk (CoVaR) and conditional Expected Shortfall (CoES) are influential systemic risk measures due to \cite{AB16}. 
In a nutshell, they assess the riskiness of a position $Y$ given that a reference position $X$ is at risk. Following \cite{GT13}, the latter conditioning event is interpreted as $X$ exceeding its Value at Risk at level $\beta$, $\VaR_\beta$. 
Hence, these risk measures can be viewed as functionals, first mapping a bivariate distribution $\rP_{X,Y}$ to the univariate conditional distribution $\rP_{Y\mid X \ge \VaR_{\beta}(X)}$ and then applying a risk measure such as $\VaR_{\alpha}$ (in the case of CoVaR) or $\ES_\alpha$ (in the case of CoES) to this univariate distribution. 
Since $\VaR_{\alpha}$ and $\ES_\alpha$ are measurable functionals from univariate distributions to $\R$, we obtain the measurability of these conditional risk measures if the map $\rP_{X,Y}\mapsto \rP_{Y\mid X \ge \VaR_{\beta}(X)}$ is measurable. 
To account for possible discontinuities in the marginal distribution of the first component $X$, we follow Remark C.1 in \cite{FisslerHoga2021} and incorporate a correction term to this map.
Hence, we formally consider
$\eta\colon \cP^2 \to \cP$, where $\cP^2$ is a set of Borel-distributions on $\R^2$ and $\cP$ is a set of such distributions on $\R$, such that for a Borel set $B\in\B(\R)$ we get  
\begin{align*}
&\eta(\rP_{X,Y})(B)=
\rP_{Y\mid X \succcurlyeq \VaR_\beta(X)}(B)\\ \nonumber
&:=
\frac{1}{1-\beta}\Big[\rP_{X,Y}\big((\VaR_{\beta}(\rP_X),\infty) \times B\big) \\ \nonumber
&\qquad + \Big(1-\beta - \rP_X \big((\VaR_{\beta}(\rP_X),\infty)\big) \Big)\rP_{X,Y}\big(\{\VaR_{\beta}(\rP_X)\} \times B\big)/\rP_X \big(\{\VaR_{\beta}(\rP_X)\}\big)\Big]
\end{align*}
For the last term, note that $\rP_X \big(\{\VaR_{\beta}(\rP_X)\}\big)=0$ implies that $1-\beta - \rP_X \big((\VaR_{\beta}(\rP_X),\infty)\big)=0$ for which case we set $0/0=0$. 
To show the $\Aa(\cP^2)-\Aa(\cP)$-measurability of $\eta$ it is sufficient to show that for any $B\in\B(\R)$, the concatenation with the evaluation map $\pi_B\circ \eta$ is $\Aa(\cP^2)-\B([0,1])$-measurable.
This measurability follows from Lemma \ref{lem:technical} and the fact that for any $C\in\B(\R)$ the map $\rP_{X,Y}\mapsto \rP(\cdot \times C)$ is $\Aa(\cP^2)-\Aa(\cP)$-measurable.
\end{example}

\begin{remark}[Continuity of statistical functionals]
	Apart from mere measurability, continuity or even differentiability properties of statistical functionals are of interest in statistics, finance and econometrics. A statistical functional is weakly continuous at $\rP \in \cP$ if the weak convergence $\rP_n \Rightarrow \rP$ of a sequence $(\rP_n)$ in $\cP$ to $\rP$ implies that $T(\rP_n) \to T(\rP)$. Let $\B(\cP)$ denote the Borel-$\sigma$-algebra generated by the weak topology on $\cP$. 
	If a statistical functional $T$ is weakly continuous at every $\rP\in \cP$ it is also $\B(\cP)$-measurable. 
	Since $\B(\cP) \subseteq \Aa(\cP) $, this readily implies  that $T$ is also measurable in the sense defined above. 
	However, apart from the Range Value at Risk the functionals discussed above are known not to be continuous at every $\rP$ under the general assumptions that we impose: For the mean, the discontinuity follows from \citet[Lemma 2.1]{HuberRonchetti2009}, while the lower $\alpha$-quantile, $\alpha\in(0,1)$, is weakly continuous at $F$ if and only if the quantile function $\gamma\mapsto q_\gamma^-(F)$ is continuous at $\alpha$, that is, if and only if $q_\alpha^-(F)=q_\alpha^+(F)$; see \citet[Lemma 21.2]{VanderVaart1998}.

	In the finance literature, statistical functionals arise from law-determined (often called law-invariant) risk measures and are often called risk functionals \citep{KratschmeSchiedETAL2014}. Assuming monotonicity with respect to first order stochastic dominance and translation equivariance, \citet[Lemma 2.1 and Corollary 2.1]{Weber2006} establishes the weak continuity of risk functionals on the class of probability distributions with compact support, implying the measurability. 
	While this applies to quantiles and weighted averages of quantiles with weight function $w$ satisfying $\int w(\gamma)\, \dd \gamma=1$, we do not require compactly supported distributions.\\	
	Generalizing \cite{Weber2006}, \cite{KratschmeSchiedETAL2014} show that monotone and translation equivariant risk functionals arising from convex risk measures on $\cP\subseteq \cQ^p = \{\rP\in\cQ \mid \int |x|^p \, \dd \rP(x)\}$ are continuous in the Wasserstein metric of order $p\in[1,\infty)$. 
	Similarly, \cite{KieselETAL2016} show the continuity of $L$-functionals in the Wasserstein metric of order $p\in[1,\infty)$. 
	Since a sequence of probability measures $(\rP_n)$ in $\cP$ converges to $\rP$ in the Wasserstein metric of order $p$ if and only if it converges weakly and $\int |x|^p \, \dd \rP_n(x) \to \int |x|^p \, \dd \rP(x)$, the continuity in the Wasserstein metric of order $p$ implies $\Aa(\cP)$-measurability. 
	This yields an alternative argument for the $\Aa(\cP)$-measurability of the lower and upper expected shortfall.
\end{remark}

\section{Elicitable functionals}
\label{sec:elicitable}

Let us now turn to general elicitable functionals. These are functionals such that the ideal forecast corresponds to a Bayes act in the sense that it minimizes an expected loss function. 
Again, let $\odom \subseteq \R$ be a measurable subset, let $\adom$ be an interval, and let $\cP$ be a family of Borel probability measures on $(\odom, \sodom)$.  
Let 
\[ 
\los : \adom \times \odom \to \R
\] 
be a scoring or loss function with the following properties:

\begin{enumerate}
	\item For each $a \in \adom$ and $\rP \in \cP$, $\los(a,\cdot)$ is $\rP$-integrable. Set
	\[ \blos(a,\rP) = \int_\odom \los(a,y)\, \dd \rP(y).\]
	We also write $\blos(a,F)$ if $F$ is a distribution function of a distribution $\rP$ in $\cP$. 
	\item For each $F \in \cP$, the function $a \mapsto \blos(a,F)$, $a \in \adom$, 
	\begin{enumerate}
		\item is continuous,
		\item has a compact, non-empty interval of minimizers, denoted by $I_\los(F)$,
		\item is monotonically decreasing to the left of $I_\los(F)$, and monotonically increasing to the right. 
	\end{enumerate}
\end{enumerate}		 
Then $\los$ is \emph{$\cP$-consistent} for the functionals $T_{\min}(F) = \min I_\los(F)$ and $T_{\max}(F) = \max I_\los(F)$ meaning that $T_{\min}(F)$ and $T_{\max}(F)$ minimize the expected loss function $a\mapsto \blos (a,F)$ for each $F\in\cP$. 
If $I_\los(F)$ is a singleton for each $F \in \cP$, then $T_{\min} = T_{\max}=:T$, then $\los$ is \emph{strictly $\cP$-consistent} for $T$, meaning that $T$ uniquely minimizes the expected loss.
A functional is \emph{elicitable} on $\cP$ if there exists a strictly $\cP$-consistent scoring function for it.

If the class $\cP$ is convex and if the continuity assumption a) on the expected loss holds, then the strict $\cP$-consistency of $\los$ for $T_{\min} = T_{\max}$ already implies properties b) and c) as shown in \citet[Proposition 2.2]{FisslerZiegel2019}, \citet[Proposition 3]{Nau1985}, \citet[Proposition 3.4]{BelliniBignozzi2015}, \citet[Proposition 11]{Lambert2019}; see also \citet[Theorem 5]{SteinwartPasinETAL2014} for a related result.
Properties b) and c) are known as \emph{order-sensitivity} or \emph{accuracy-rewarding} in the literature; see \cite{LambertETAL2008} and the references above.
As noted in the discussion below Proposition 2.2 in \cite{FisslerZiegel2019}, a sufficient condition for the continuity of the expected loss $a\mapsto \blos(a,F)$ for each $F\in\cP$ are conditions a), b) and c) on the level of the loss $a\mapsto \los(a,y)$ itself for each $y\in\odom$.
The continuity of $\los(a,y)$ in its first argument is a common regularity condition; see e.g.~\cite{Gneiting2011}.

Define the entropy, uncertainty, or Bayes risk of the $\cP$-consistent loss $\los$ for $T$ as $T_{\ent}(F) = \blos\big(T(F),F\big)$, $F\in\cP$.

\begin{theorem}\label{thm:main}
	Under the above assumptions, the functionals $T_{\min}$, $T_{\max}$ and $T_{\ent}$ are $\Aa(\cP) - \B$-measurable. 
\end{theorem}	
\begin{proof}
First consider $T_{\min}$.
By Example \ref{ex:moments}, for each $a\in\adom$ the functional $F\mapsto \blos(a,F)$ is $\Aa(\cP) - \B$-measurable. We then argue similarly as in Example \ref{ex:quantiles}: Let $a \in \adom$ and $F \in \cP$. We claim that under our assumptions, $T_{\min}(F) > a$ holds if and only if there exists $b \in \adom \cap \Q$ with $b > a$ such that $\blos(a,F) > \blos(b,F)$. 
Indeed, by definition of $T_{\min}$, $T_{\min}(F) > a$ holds if and only if 
$a$ is to the left of $I_\los(F)$. By monotonicity and continuity of $\blos(\cdot,F)$, 
this holds if and only if there exists $b \in (a,\infty)\cap \Q \cap \adom$ such that $\blos(a,F) > \blos(b,F)$.  
Therefore,
\begin{align*}
	\{F\in\cP \mid  T_{\min}(F) >a \} & = \bigcup_{b \in (a,\infty)\cap \Q \cap \adom} \{ F\in \cP \mid \blos(a,F) - \blos(b,F) >0\},
\end{align*}
showing measurability of $T_{\min}(F)$. The arguments for $T_{\max}$ work similarly.

For $T_{\ent}$, we obtain the measurability upon noting that for any $x\in\R$
\begin{align*}
	\{F\in\cP \mid  T_{\ent}(F) < x \} & = \bigcup_{a \in  \Q \cap \adom} \{ F\in \cP \mid \blos(a,F) < x\}.
\end{align*}
\end{proof}	

\begin{example}[Expectiles and generalized quantiles]

On the class of distributions with finite mean, \cite{NeweyPowell1987} introduced the $\tau$-expectile of a distribution $\rP$, $e_\tau(\rP)$, $\tau\in(0,1)$, as the unique solution to the equation
\[
\tau \int_{(x,\infty)} (y-x) \,\mathrm{d}\rP(y) = (1-\tau) \int_{(-\infty,x]} (x-y)  \,\mathrm{d}\rP(y)
\]
in $x$. On the class of square-integrable distributions, $e_\tau$ is elicitable with the asymmetric piecewise quadratic loss $\los_\tau(a,y) = |\ind\{y\le a\} - \tau|(a-y)^2$ as a strictly consistent loss. Hence, Theorem \ref{thm:main} establishes the measurability of $e_\tau$ on this class. 
For the measurability on the class of integrable distributions, one can consider the slightly modified loss function $\widetilde \los_\tau(a,y) = \los_\tau(a,y) - \los_\tau(0,y)$. 

\cite{BelliniKlarETAL2014} study \emph{generalized} $\tau$-quantiles which arise as minimizers of the expected loss 
$L_{\tau,\phi_1,\phi_2} (a,y) = \ind\{y\le a\} (1 - \tau)\phi_1(|a-y|) + \ind\{y> a\}\tau \phi_2(|a-y|)$ for two convex and strictly increasing functions $\phi_1, \phi_2 : [0,\infty)\to[0,\infty)$. Clearly, if $\phi_1$ and $\phi_2$ are both the absolute function (the squared function) the $\tau$-quantile ($\tau$-expectile) arises. \cite{ElliottETAL2005} studied the situation of general power functions. 
Theorem \ref{thm:main} then establishes the $\Aa(\cP) - \B$-measurability of these generalized $\tau$-quantiles.
\end{example}

\begin{example}[Variance and Expected Shortfall]\label{ex:varianceES}
On the class of square-integrable distributions $\cP$, the squared loss $L(a,y) = (a-y)^2$ is a strictly $\cP$-consistent loss for the mean functional. 
The corresponding entropy functional $T_{\ent}$ is the variance functional.
Therefore, Theorem \ref{thm:main} establishes the $\Aa(\cP) - \B$-measurability of the variance as an alternative argument to the direct application of Example \ref{ex:moments}.

For the lower or the upper Expected Shortfall at level $\alpha$ it is well known \cite[Lemmas 2.3 and 3.3]{EmbrechtsWang2015} that both versions can be written as the entropy of a consistent loss function for $q_\alpha^-$.
Therefore, Theorem \ref{thm:main} establishes an alternative argument for the measurability to that provided in Example \ref{ex:average quantiles}.
\end{example}

An alternative approach to establish the measurability of a functional is to exploit identifiability instead of elicitability. 
A functional is called identifiable if there exists a moment or identification function $\iden\colon\adom\times\odom\to\R$ such that the functional is the unique zero of the expected identification function. 
Indeed, under some technical assumptions, elicitability and identifiability are equivalent for one-dimensional functionals \citep[Theorem 5]{SteinwartPasinETAL2014}, and concerning measurability, one ends up with similar results and examples as considered in this section.

\begin{remark}\label{rem:continuity and elicitability}
Proposition 3.7 in \cite{BelliniBignozzi2015} establishes $\psi$-weak continuity of elicitable functionals on classes of probability measures with compact support under additional regularity conditions on the corresponding strictly consistent loss function. 
$\psi$-weak continuity has been discussed in \cite{KratschmeSchiedETAL2014} and generalizes continuity in the Wasserstein metric. It implies $\Aa(\cP)$-measurability in the sense studied in this paper.
\end{remark}

\section{Measurability of scoring rules}
\label{sec:scoring rules}

In this section we briefly discuss measurability of scoring rules, which are used to compare probabilistic forecasts. 
Consider maps
\begin{equation}\label{eq:scorerule}
	\rS : \cP \times \odom \to (-\infty, \infty],	
\end{equation}
such that for each $\rP \in \cP$, $\rS(\rP;\cdot)$ is Borel-measurable and that for each $\rQ \in \cP$, 
\[ \E_\rQ\big[\rS(\rP;Y) \big] = \int_\odom \rS(\rP;y)\, \dd \rQ(y)  \]
exists.
Note that we allow $\infty$ in the range of $\rS$ in \eqref{eq:scorerule}  to treat the common logarithmic score, see Example \ref{ex:log score}.
$\rS$ is called a strictly $\cP$-proper scoring rule if
$$ \E_\rQ\big[\rS(\rQ;Y) \big] <  \E_\rQ\big[\rS(\rP;Y) \big] \qquad \text{for all } \ \rP, \rQ \in \cP,\ \rP \not= \rQ,$$
meaning that forecasting the true distribution $\rQ$ uniquely minimizes the expected score.  

Recall that the ideal probabilistic forecast based on an information set $\F$ is the conditional distribution $\Pb_{Y | \F}$. 
That means in a forecasting situation, with available information $\F$ (e.g.~given by regressors), it is natural that the forecasts have the form of Markov kernels from $(\Omega, \F)$ to $(\odom, \sodom)$.
Inserting a Markov kernel $\kappa$ for $\rP$ leads to the map
\begin{equation}\label{eq:expvalue}
 \omega \mapsto \rS\big(\kappa(\omega,\cdot);Y(\omega)\big),
\end{equation} 
and averages of such expressions, which estimate expected values, are used when comparing different forecasts. 
But in order to form the expected value in \eqref{eq:expvalue} we in particular require measurability, i.e.~we need that \eqref{eq:expvalue} is an $\Aa - \B$-measurable random variable. Since $\omega \mapsto \big(\kappa(\omega,\cdot);Y(\omega)\big)$ is $\Aa - \Aa(\cP) \otimes \sodom$-measurable,
this will be guaranteed if we assume that $\rS$ in \eqref{eq:scorerule} is $\Aa(\cP) \otimes \sodom - \B$-measurable.

\begin{lemma}\label{lem:measscore}
	Suppose that $\odom$ is a one-dimensional interval, 
	and that 
	\begin{enumerate}
		\item for each $y \in \odom$, the map $\rS(\cdot; y) : \cP \to (-\infty, \infty]$ is $\Aa(\cP)  - \B$-measurable,
		\item for each $\rP \in \cP$, the map $\rS(\rP; \cdot)$ is right-continuous (or left-continuous). 
	\end{enumerate}
	Then $\rS$ in \eqref{eq:scorerule} is $\Aa(\cP) \otimes \sodom - \B$-measurable.
\end{lemma}
\begin{proof}	
	For a finite interval, say $\odom = (0,1]$, for $n \in \N$, by (i) the map 
	$$ \rS_n(\rP,y) = \sum_{k=1}^{2^n}\, \rS(\rP; k/2^n)\, \ind_{((k-1)/2^n, k/2^n]}(y)$$
	is $\Aa(\cP) \otimes \sodom - \B$-measurable and by (ii), $\rS_n(\rP,y) \to \rS(\rP,y)$ for $n \to \infty$. An infinite interval can be treated analogously by extending the sum to an infinite one. 
\end{proof}	
\begin{remark}
	The argument in Lemma \ref{lem:measscore} can be used to obtain a similar result for multivariate observations, where right-continuity is defined, as for multivariate distribution functions, by taking limits from above in each coordinate. However, the form of the observation domain $\odom$ is also an issue, and to proceed as above we require that $\odom$ is a possibly infinite hyperrectangle.    
\end{remark}
\begin{example}[Continuous ranked probability score]
	The continuous ranked probability score (CRPS) of a distribution $F$ with finite first moment is 
	\begin{equation}\label{eq:brierscore}
		\CRPS(F,y) = \int_{- \infty}^{\infty} \big(F(z) - \ind\{y \leq z\} \big)^2\, \mathrm{d} z,\quad F \in \cP,
	\end{equation}
	see \cite{GneitingRaftery2007}.
	It has an alternative representation in terms of an integrated quantile loss, and also as
	\begin{align}\label{eq:crpsexpec}
	\CRPS(F,y) & = \E_F\big[ \big|y- X \big|\big] - \frac{1}{2}\, \E_F \big[\big|X'- X \big|\big],
	\end{align}
	where $X$ and $X'$ are independent copies with distribution $F$. 
	From both \eqref{eq:crpsexpec} and \eqref{eq:brierscore} one can readily see that $\CRPS(F,y)$ is continuous as a function of $y$, invoking the dominated convergence theorem.	
	Further, for fixed $y$, \eqref{eq:crpsexpec} implies measurability of $\CRPS(F,y)$ as a map in $F$, while from \eqref{eq:brierscore} the measurability can be obtained with a Riemann sum approximation similarly as in Example \ref{ex:average quantiles}. Similar arguments can be made and hence the measurability is obtained for weighted versions of the CRPS as introduced in \cite{GneitingRanjan2011} and \cite{HolzmannKlar2017}.
\end{example}

\begin{example}[Logarithmic score and local scoring rules]
\label{ex:log score}
	Assume that $\cP$ is a family of right-continuous Lebesgue densities on the real line. We write $\rp \in \cP$ and denote the distribution function associated with $\rp$ by $F_\rp$. 
	The logarithmic score is 
	\[
	 \LOGS(\rp,y) = - \log \rp(y),\qquad \rp \in \cP,
	\]
	see \cite{Good1952} and \cite{GneitingRaftery2007}.
	
	To apply Lemma \ref{lem:measscore}, observe that $\LOGS(\rp,y)$ is right-continuous as a function of $y$ by our assumption on $\rp$ and continuity of the logarithm. 
	For the measurability in $\rp$ for fixed $y$, note that by right-continuity, 
	$$\lim_{n \to \infty} \log\Big(n\, \big(F_\rp(y+1/n) - F_\rp(y)\big) \Big) = \log \rp(y),$$
	showing measurability, since the evaluation of the distribution function is measurable. Measurability extends to weighted versions of the logarithmic score as in \cite{DiksETAL2011} and in \cite{HolzmannKlar2017}, and under appropriate assumptions on higher derivatives to higher-order local scoring rules \citep{EhmGneiting2012, ParryDawidLauritzen2012} such as the Hyv\"arinen score \citep{Hyvaerinen2005}.

\end{example}

\section*{Acknowledgements}

We would like to thank Johannes Resin for valuable feedback on an earlier version of the paper.

\bibliographystyle{apalike}

\end{document}